\documentclass{amsart}

\usepackage{amssymb}
\usepackage[all]{xy}
\usepackage{hyperref}

\usepackage{enumitem}   

\setlist[enumerate]{itemsep=.2em,topsep=.2em,leftmargin=1.25em,itemindent=2.0em}

\newtheorem{thm}{Theorem}

\newtheorem{lem}[thm]{Lemma}
\newtheorem{cor}[thm]{Corollary}

\theoremstyle{definition}

\newtheorem{say}[thm]{}
\newtheorem{exmp}[thm]{Example}

\newtheorem{ques}[thm]{Question}    

\newtheorem{rem}[thm]{Remark}
\newtheorem{rems}[thm]{Remarks} 

\newtheorem*{ack}{Acknowledgments}      

\newtheorem{defn-thm}[thm]{Definition--Theorem}  
\newtheorem{defn-lem}[thm]{Definition--Lemma}  

\theoremstyle{remark}

\let \cedilla =\c
\renewcommand{\c}[0]{{\mathbb C}}  

\renewcommand{\o}[0]{{\mathcal O}} 

\renewcommand{\a}[0]{{\mathbb A}}

\newcommand{\p}[0]{{\mathbb P}}

\newcommand{\qtq}[1]{\quad\mbox{#1}\quad}

\newcommand{\pics}[0]{\operatorname{\mathbf{Pic}}}

\newcommand{\pico}[0]{\operatorname{\mathbf{Pic}}^{\circ}}

\newcommand{\ext}[0]{\operatorname{Ext}}    
\newcommand{\Hom}[0]{\operatorname{Hom}}

\newcommand{\onto}[0]{\twoheadrightarrow}

\newcommand{\depth}[0]{\operatorname{depth}} 
\newcommand{\tsum}[0]{\textstyle{\sum}}

\def\into{\DOTSB\lhook\joinrel\to}

\newcommand{\leteq}{\colon\!\!\!=}

\def\loccoh#1.#2.#3.#4.{H^{#1}_{#2}(#3,#4)}

\DeclareMathAlphabet{\mathchanc}{OT1}{pzc}%
                                {m}{it}

\newcommand{\sExt}[0]{{\mathchanc{Ext}}}

\newcommand{\sym}[0]{\operatorname{Sym}}

\usepackage[all]{xy}\xyoption{dvips}

\newcommand{\myR}{{\mathchanc{R}\!}}

\newcommand{\kdot}{{{\,\begin{picture}(1,1)(-1,-2)\circle*{2}\end{picture}\,}}}
\newcommand{\lotimes}{\overset{L}{\otimes}}
\newcommand{\myL}{{\mathchanc{L}\!}}
\newcommand{\blank}{\mbox{---}}
\newcommand{\sHom}[0]{{\mathchanc{Hom}}}

\begin{document}
\bibliographystyle{amsalpha}


 \title{Higher direct images of dualizing sheaves {III}}
\author{J\'anos Koll\'ar}
\email{kollar@math.princeton.edu}
\address{\vskip-1.75em Department of Mathematics, Princeton University, Fine
  Hall, Washington Road, Princeton, NJ 08544-1000, USA} 
\author{S\'andor J Kov\'acs}
\email{skovacs@uw.edu}
\address{\vskip-1.75em University of Washington, Department of Mathematics, Box 354350,
Seattle, WA 98195-4350, USA}
 \begin{abstract}  We show that for flat morphisms between varieties with rational singularities, the higher direct images of the structure sheaf are locally free.  As a consequence, the identity component of the relative  Picard scheme is a smooth algebraic group scheme.
   
       \end{abstract}

 \maketitle

 It is frequently useful to show that certain cohomology groups are invariant under flat deformations.
 The best known example is the deformation invariance of the  groups
 $H^q(X, \Omega_X^p)$ for smooth projective varieties in characteristic 0.
 Much less is known for singular varieties. The notion of Du~Bois singularities was developed to guarantee the deformation invariance of the 
 cohomology groups of the structure sheaf \cite{dub-jar}.

 Given a variety $X_0$ and a sheaf $F_0$, most proofs of  the  deformation invariance of some $H^i(X_0, F_0)$  rely on local properties of $X_0$ and $F_0$; see for example \cite{k-db2, k-modbook, gabber2024cohomologicalflatnessdiscretevaluation} for typical examples.

 Our aim is to obtain deformation invariance results assuming  that the singularities of  the total space of the deformation are mild, but  knowing   little about special fibers.
 Working with varieties over a field of characteristic 0 throughout, 
 the main result is the following.

 \begin{thm} \label{lf.thm} Let $g:X\to S$ be a flat, projective morphism.   Assume that
   $X$ has rational singularities.
      Then
   \begin{enumerate}
   \item $\myR^ig_*\o_X$ and $\myR^ig_*\omega_{X/S}$ are locally free and commute with base change.
 \item $\myR g_*\o_X\simeq \tsum \myR^ig_*\o_X[-i]$.
 \item $\myR g_*\omega_{X/S}\simeq \tsum \myR^ig_*\omega_{X/S}[-i]$.
   \item $\myR^ig_*\omega_{X}\cong \omega_S\otimes \myR^ig_*\omega_{X/S}$.
     \end{enumerate} 
 \end{thm}

 \begin{rems}
  $S$ has rational singularities by \cite{BoutotRational}.
 
   We do not assume that $g$ has connected fibers.

 If $X$ is only assumed CM, the properties (\ref{lf.thm}.1--4) almost imply each other;  see Theorem~\ref{locally.omega.thm} for the a precise formulation.

 For Lagrangian fibrations of  hyperk\"ahler varieties,
\cite{MR2130616} determines the sheaves $\myR^ig_*\o_X$ and $\myR^ig_*\omega_{X/S}$; they are locally free. This was used in 
\cite{kim2024neronmodelhigherdimensionallagrangian} to show that 
$\pics^\circ(X/S)\to S$ is a smooth algebraic group for Lagrangian fibrations. Our main aim was to extend Kim's  result to  more general settings, especially to the
Abelian fibrations considered in \cite{k-neron, k-abtwist}. 
 
 As we discuss in Theorem~\ref{hdi2.thm},  \cite{k-hdi2} 
 implies that  (\ref{lf.thm}.1--4) hold over a large open subset
 $S^\circ\subset S$ (even if $g$ is not assumed flat). Thus the main  assertion is that if $g$ is flat, then these hold everywhere.

 There are similar results if  $X$ has some non-rational singularities, see Theorem~\ref{iso.nonrta.rem}. However,  already (\ref{lf.thm}.1) can fail if $X$ is only assumed CM.
Example~\ref{exmp.1} shows that for $i\geq 2$, the $h^i(X_s, \o_{X_s})$ can jump for families whose members are normal hypersurfaces in a smooth family of varieties.
It is not clear what an optimal result would be.

In positive characteristic, the $h^i(X_s, \o_{X_s})$ can jump even in smooth, projective families, as shown by Enriques surfaces in characteristic 2 \cite{MR491720}.

It is quite likely that the results also hold for rather general base schemes over a field of characteristic 0. Our proof relies heavily on
\cite{k-hdi1, k-hdi2} which need $S$ to be a variety.
The torsion freeness of the sheaves  $R^ig_*\omega_X$ \cite[2.1]{k-hdi1}  has been extended to general base schemes in \cite{MR4945100, mur-van}. However, in Theorem~\ref{hdi2.thm} we rely more heavily on variations of Hodge structures. It is not obvious how to extend these result to other base schemes.
\end{rems}

\medskip 
 
The proof is given in Paragraph~\ref{lf.thm.pf}, but first we discuss some applications. We start with the structure of the relative Picard scheme.
As explained in \cite[Sec.2.2]{kim2024neronmodelhigherdimensionallagrangian},
the combination of \cite[Thm.7.3]{Artin69b}, \cite[Thm.8.4.1]{blr},  and
Prop.1.6  and Thm.3.1 in \cite[Exp.VI$_B$]{sga3} yields the following.

   \begin{cor} \label{lf.cor.1}
    Let   $g:X\to S$ be as in Theorem~\ref{lf.thm}. Then
    $\pics^\circ(X/S)\to S$ is a smooth algebraic group scheme, which
    commutes with all base changes. \qed
    \end{cor}

   \begin{rems}
     Cases when $X$ has some non-rational singularities are treated in Corollary~\ref{lf.cor.1.nr}.

   These proofs give that $\pics^\circ(X/S)\to S$ is an algebraic space.
      Quite likely, it is also quasi-projective. 
    \end{rems}

      For non-flat morphisms, Theorem~\ref{lf.thm} gives the following.

  \begin{cor} \label{lf.cor.2} Let $g:X\to S$ be a projective, equidimensional morphism.
    Assume that $X$  has rational singularities.  Then
    \begin{enumerate}
    \item   The sheaves $\myR^ig_*\o_X$  and $\myR^ig_*\omega_X$ are maximal  CM \cite[\href{https://stacks.math.columbia.edu/tag/00NF}{Tag
00NF}]{stacks-project}, 
    \item $\myR g_*\o_X\simeq \tsum \myR^ig_*\o_X[-i]$, and 
    \item $\myR g_*\omega_{X}\simeq \tsum \myR^ig_*\omega_{X}[-i]$.
    \end{enumerate}
  \end{cor}

The proof is in  Paragraph~\ref{lf.cor.2.pf}.

  Note the change from  $\omega_{X/S}$ in Theorem~\ref{lf.thm} to $\omega_{X}$ here. We do not know what happens for
    the $\myR^ig_*\omega_{X/S}$.

    Even if $g$ is not equidimensional, the sheaves
    $\myR^ig_*\omega_X$ are all torsion-free by \cite[2.1]{k-hdi1},
     and $\myR^1g_*\o_X$  is $S_2$ 
    by \cite[7.8]{k-hdi1} if $g$ has  connected fibers.  The other  $\myR^ig_*\o_X$ are usually not even torsion-free, see Example~\ref{h2.tor.exmp}
    
    (\ref{lf.cor.2}.3) holds by \cite[3.1]{k-hdi2}.
    \medskip

  \subsection*{The proof of  Theorem~\ref{lf.thm}}{\ }
 \medskip

 The proof  relies on the following, which  is one of  the main theorems of \cite{k-hdi2}, but it is not stated explicitly enough for our current purposes.

  An open subset $S^\circ\subset S$ is called {\it large}
  if $S\setminus S^\circ$ has codimension $\geq 2$.

 \begin{thm} \label{hdi2.thm} Let $g:Y\to S$ be a  projective morphism. Assume that $Y$ has  rational singularties and  $S$ is normal. Let $H$ be a  member of a basepoint-free linear system on $Y$. Then there is a large open subset
   $S^\circ\subset S$ such that the following hold for  $g^\circ: Y^\circ:=g^{-1}(S^\circ)\to S^\circ$ and $H^\circ:=Y^\circ\cap H$.
   \begin{enumerate}
\item All the sheaves  $\myR^ig^\circ_*\o_{Y^\circ}, \myR^ig^\circ_*\o_{Y^\circ}(-H^\circ), 
  \myR^ig^\circ_*\omega_{Y^\circ}$, and $ \myR^ig^\circ_*\omega_{Y^\circ}(H^\circ)$ are locally free.
\item Each of the natural maps
  $\myR^ig^\circ_*\o_{Y^\circ}\to \myR^ig^\circ_*\o_{H^\circ}$ 
  and
 $ \myR^ig^\circ_*\omega_{H^\circ}\to
 \myR^{i+1}g^\circ_*\omega_{Y^\circ}$
 is a composite of a split surjection followed by a split injection.
 \item $\myR g^\circ_*\o_{Y^\circ}\simeq \tsum\myR^ig^\circ
   _*\o_{Y^\circ}[-i]$ and
   $\myR g^\circ_*\omega_{Y^\circ}\simeq \tsum\myR^ig^\circ
 _*\omega_{Y^\circ}[-i]$.
   \end{enumerate}
   \end{thm}

 Proof.    Let $r:Y'\to Y$ be a resolution of singularities and $g':Y'\to S$ the composite.
  Then $\o_Y=\myR r_*\o_{Y'}$ and
 $\omega_Y=\myR r_*\omega_{Y'}$.
 Thus $\myR^ig_*\o_Y=\myR^ig'_*\o_{Y'}$, and similarly for
 $\myR^ig_*\o_Y(-H), \myR^ig_*\omega_Y, \myR^ig_*\omega_Y(H)$.
 It is thus enough to prove the claims when $Y$ is smooth.

 Let $D\subset S$ be the locus of singular fibers, and
 $Z\subset D$ the smallest closed subset  such that
 $D\setminus Z$ is a smooth divisor in $S \setminus Z$.
 Set $S^\circ:=S \setminus Z$.

 The  sheaves
  $\myR^ig^\circ_*\o_{Y^\circ}, 
 \myR^ig^\circ_*\omega_{Y^\circ}$ are locally free by
 \cite[2.6]{k-hdi2}.

 Let $\pi:\bar Y\to Y$ be a double cover ramified along a general member of $|2H|$.
 As in \cite[2.44]{kk-singbook}, we see that
 $\o_Y(-H)$ is a direct summand of
 $\pi_*\o_{\bar Y}$, and
  $\omega_Y(H)$ is a direct summand of
 $\pi_*\omega_{\bar Y}$. Applying  \cite[2.6]{k-hdi2}
 to $\bar Y\to S$ (and possibly shrinking $S^\circ$)
 we get that $\myR^ig^\circ_*\o_{Y^\circ}(-H), 
 \myR^ig^\circ_*\omega_{Y^\circ}(H)$ are also locally free.

 The main theorem of \cite[p.174]{k-hdi2} says that all the sheaves
 in (\ref{hdi2.thm}.1) are determined by the variations of the Hodge structures on the cohomology groups   $H^i(Y_b, \c)$ and $H^i(\bar Y_b, \c)$ of the fibers  for 
 $b\in S \setminus D$,
 and this correspondence respects direct sums.
Variation of pure Hodge structures form a semi-simple category by \cite{MR0498551}, hence maps between them 
are composites of a split surjection followed by a split injection.

The coherent restriction map $H^i(Y_b, \o_{Y_b})\to H^i(H_b, \o_{H_b})$ is induced by the topological restriction map $H^i(Y_b, \c)\to H^i(H_b, \c)$, which proves 
(\ref{hdi2.thm}.2) for 
$\myR^ig^\circ_*\o_{Y^\circ}\to \myR^ig^\circ_*\o_{H^\circ}$.
Serre duality now gives the claim for
$ \myR^ig^\circ_*\omega_{H^\circ}\to
 \myR^{i+1}g^\circ_*\omega_{Y^\circ}$.

 Finally (\ref{hdi2.thm}.3) for $\omega_Y$ is a special case of
 \cite[3.1]{k-hdi2}. Since  the $\myR^{i+1}g^\circ_*\omega_{Y^\circ}$ are locally free, duality gives that
$\myR g^\circ_*\o_{Y^\circ}\simeq \tsum\myR^ig^\circ
   _*\o_{Y^\circ}[-i]$. 
 \qed

 \begin{rem} \label{hdi2.thm.rem} Note that while  every sheaf in the  long exact sequences 
  $$
  \begin{array}{l}
 0\to g^\circ_*\omega_{Y^\circ}\to g^\circ_*\omega_{Y^\circ}(H^\circ)\to g^\circ_*\omega_{H^\circ}\to
 \myR^1g^\circ_*\omega_{Y^\circ}\to \cdots,  \qtq{and}\\[1ex]
 0\to g^\circ_*\o_{Y^\circ}(-H^\circ)\to g^\circ_*\o_{Y^\circ}\to g^\circ_*\o_{H^\circ}\to
 \myR^1g^\circ_*\o_{Y^\circ}(-H^\circ)\to \cdots
  \end{array}
  \eqno{(\ref{hdi2.thm.rem}.1)}
  $$
  comes from a variation of pure Hodge structures,
  for $g^\circ_*\omega_{Y^\circ}(H^\circ) $ and $ g^\circ_*\o_{Y^\circ}(-H^\circ)$
  this identification is somewhat artificial,  and some of the maps between the coherent higher direct images  naturally arise from maps of mixed Hodge structures.
Since variations of  mixed Hodge structures  do not  form a semi-simple category, 
 the other maps in the sequences (\ref{hdi2.thm.rem}.1) need not   split.

 As an example, let  $E_i$ be  elliptic curves and $L_i$ line bundles of degrees $d_i\geq 2$. Set $S:=E_1\times E_2$, $\pi:S\to E_1$ the projection, and  $C\in |L_1\boxtimes L_2|$. Then we get
 $$
 \o_{E_1}\cong \pi_*\omega_{S/E_1}\to \pi_*\omega_{S/E_1}(C)\cong  \oplus_1^{d_2}L_1,
 $$
 which has no splitting.
 \end{rem}

  \begin{say}[Proof of Theorem~\ref{lf.thm}]\label{lf.thm.pf}
 The claims in (\ref{lf.thm}.1) are local, we may thus assume that $S$ is affine.
 The proof is by induction on the relative dimension.
 If the relative dimension is 0, then
 $g_*\o_X$ is flat, hence locally free, and
 $g_*\omega_{X/S}$ is its dual, hence also locally free.

 For relative dimension $\geq 1$, let $H\subset X$ be a general, sufficiently ample divisor.
 Then $H$ also has rational singularities, and $g:H\to S$ is also
 flat. 
 Pushing forward
 $$
 0\to \omega_{X/S}\to \omega_{X/S}(H)\to \omega_{H/S}\to 0
 \eqno{(\ref{lf.thm.pf}.1)}$$
 we get isomorphisms
 $$
 \myR^ig_*\omega_{H/S}\cong  \myR^{i+1}g_*\omega_{X/S} \qtq{for } i\geq 1,
 \eqno{(\ref{lf.thm.pf}.2)}$$
 and an exact sequence
 $$
 0\to g_*\omega_{X/S}\to g_*\omega_{X/S}(H)\to g_*\omega_{H/S}\to
 \myR^1g_*\omega_{X/S}\to 0.
 \eqno{(\ref{lf.thm.pf}.3)}
 $$
 Induction and (\ref{lf.thm.pf}.2) show that the
 $\myR^{i+1}g_*\omega_{X/S}$ are locally free for $i\geq 1$.

 In (\ref{lf.thm.pf}.3) 
 we know that $g_*\omega_{X/S}(H)$ is locally free  (since $\omega_{X/S}$ is flat over $S$) and $g_*\omega_{H/S}$ is locally free  (by induction).
 Also, $g_*\omega_{X/S}$ is  $S_2$ (since  $\omega_{X/S}$ is $S_2$ and $g$ is equidimensional).

 We apply Theorem~\ref{hdi2.thm}  to
 get $g^\circ: X^\circ\to S^\circ$ such  that 
$
\delta^\circ: g^\circ_*\omega_{H^\circ/S^\circ} \onto \myR^1g^\circ_*\omega_{X^\circ/S^\circ}
$
 has a splitting.

 A splitting of an $S_2$  sheaf over a large open set extends to a splitting by (\ref{split.lem}.2). Thus $g_*\omega_{H/S}$ decomposes as
 $$
 g_*\omega_{H/S}\cong E_1+E_2,
 $$
  where  $\delta^\circ$ is $0$ on $E_1^\circ$, and maps $E_2^\circ$ isomorphically onto
 $\myR^1g^\circ_*\omega_{X^\circ/S^\circ}$.
  Thus $\myR^1g_*\omega_{X/S}$ is  isomorphic to the direct sum of $E_2$ and a  (possibly nonzero) torsion quotient of  $E_1$. Let us now posit that  the following special case of
  (\ref{lf.thm}.4) holds in our situation.

  \medskip
      {\it Assumption \ref{lf.thm.pf}.4.}
      $\myR^1g_*\omega_{X}\cong  \omega_S\otimes \myR^1g_*\omega_{X/S}$.

  \medskip
  
 If this holds, then, by the Nakayama lemma, a nonzero torsion summand
 of $\myR^1g_*\omega_{X/S}$ 
 would give a nonzero torsion summand  of $\myR^1g_*\omega_{X}$, but the latter is torsion-free  by \cite[2.1]{k-hdi1}.
 Thus $\myR^1g_*\omega_{X/S}\cong E_2$ is locally free, and then so is
$g_*\omega_{X/S}$.

In particular,  $s\to h^i(X_s, \omega_{X_s})$ is locally constant for every $i$. Since each $X_s$ is CM, each 
$s\to h^i(X_s, \o_{X_s})$ is locally constant for every $i$ by Serre duality. Thus the $\myR^ig_*\o_{X}$ are locally free by Grauert's theorem.
(See also  (\ref{dual.elem.say}.1) for a stronger version of this argument.)

For the proof of (\ref{lf.thm}.2--4) $S$ is any variety (not necessarily affine). 
Then   (\ref{lf.thm}.2) follows from  (\ref{lf.thm}.1)
by \cite[3.12]{k-hdi2}. Duality then  gives (\ref{lf.thm}.3).

 Finally (\ref{lf.thm}.3) and (\ref{cor:projection-formula}.1)  imply (\ref{lf.thm}.4).

 It remains to show that Assumption~\ref{lf.thm.pf}.4 holds.
 The easy case is when  $S$ is smooth, more generally, when
 $\omega_S$ is locally free.  (This seems to be the main case for applications.)
 Then  $\omega_X\cong\omega_{X/S}\otimes g^*\omega_S$, so 
 (\ref{lf.thm.pf}.4) holds by the  projection formula.  

If $\omega_S$ is not locally free, the proof is more complicated; see Paragraph~\ref{verify.assump}. 
  \qed
  \end{say}

  \begin{rem} \label{split.lem}
    Let $S$ be a scheme, $Z\subset S$ a closed subset,  and
    $U:=S\setminus Z$ with injection $j:U\into S$.

    Let $F,G$ be coherent sheaves on $S$ and assume that $\depth_ZG\geq 2$. Note that the latter holds if $G$ is $S_2$, and $Z\subset S$ has codimension $\geq 2$.

    Then
    $\Hom_S(F, G)=\Hom_U(F|_U, G|_U)$. In particular
    \begin{enumerate}
    \item  $\ext^1_S(F, G)\into\ext^1_U(F|_U, G|_U)$ is injective\footnote{In version 1 of this preprint we incorrectly claimed that this holds for all $\ext^i$. We thank B.~Bhatt for pointing  this out.}, and
    \item if $G|_U\cong G_1\oplus G_2$, then
      $G\cong j_*G_1\oplus j_*G_2$.
    \end{enumerate}
    \end{rem}

\begin {say}[Proof of Corollary~\ref{lf.cor.2}]\label{lf.cor.2.pf} Note that $X$ is normal, hence we may harmlessly replace $S$ with its normalization.

  For (\ref{lf.cor.2}.1), we may assume that $S$ is affine, and choose a finite  $\pi:S\to S'$ where $S'$ is smooth.  Note that $g':=\pi\circ g$ is flat since $S'$ is smooth.
  We can thus apply Theorem~\ref{lf.thm} to  $g':X\to S'$, hence  $\myR^ig'_*\o_X$ is locally free, hence CM. 

  Note also that 
  $\pi_*\myR^ig_*\o_X=\myR^ig'_*\o_X$, so  $\myR^ig_*\o_X$ is also CM, and its support equals $g(X)$.
The same argument works for $\omega_X$.

Once (\ref{lf.cor.2}.1) holds, 
\cite[3.12]{k-hdi2} implies (\ref{lf.cor.2}.2),  and
(\ref{lf.cor.2}.3) was proved in \cite[3.1]{k-hdi2}. \qed
\end{say}

  \medskip
 A more elementary consequence of  Theorem~\ref{lf.thm} is the following.

 \begin{cor}\label{cor.0} Let $Y$ be a projective variety with rational singularities, and $|L|$ a basepoint-free linear system.
   Then  the $h^i(D, \o_D)$ are independent of $D\in |L|$.
 \end{cor}

 Proof.  If $L$ is ample, this follows from Kodaira's vanishing, but the general case seems more subtle.

 The universal family  $X\subset Y\times |L|$ is a projective space bundle over $Y$, hence it has rational singularities.  Thus
        Theorem~\ref{lf.thm} applies to the projection to $|L|$. 
 (If $Y$ is smooth, this also follows from  \cite[8.2]{k-hdi1}.) \qed
 \medskip

 \subsection*{Examples}{\ }
         \medskip

        Note that the $D$  in Corollary~\ref{cor.0} are CM, but can be reducible and nonreduced.
        Also, the total space of a  given deformation of $D$ need not have rational singularities, but the universal deformation does, and
        Theorem~\ref{lf.thm} applies to the  universal deformation.
       This raises the possibility that  Theorem~\ref{lf.thm} could also hold if the  fibers of $g:X\to S$ have only hypersurface singularities.
       The next example shows that  this is not the case, hence
        Theorem~\ref{lf.thm} cannot be a consequence of a result that assumes   properties of a single fiber only.

 \begin{exmp}\label{exmp.1}   Let  $S_0:=(x_1^4+x_2^4+x_3^4=0)\subset \p^3$ be a quartic cone, and $E$ an elliptic curve.
   We describe 2 types of deformations of $Y_0:=S_0\times E$.

   (\ref{exmp.1}.1)  Deforming $S_0$ as a quartic surface gives deformations $Y_t:=S_t\times E$.  Here Theorem~\ref{lf.thm} applies, and the
   $h^i(Y_t, \o_{Y_t})$ are independent of $t$.

   (\ref{exmp.1}.2) In Example~\ref{exmp.2} we describe another flat family  $\{Y_L: L\in \pico(E)\}$, where each $Y_L$ is an $S_0$-bundle over $E$, and $Y_0=Y_{\o_E}$. We check that
   $$
   h^2(Y_L, \o_{Y_L})= h^3(Y_L, \o_{Y_L})=
   \left\{
   \begin{array} {l}
     1 \qtq{if} L\cong \o_E, \qtq{and}\\
     0 \qtq{if} L\not\cong \o_E.
   \end{array}
   \right.
   $$

   (\ref{exmp.1}.3)  Let $S$ be any scheme, and $g:X\to S$ a  projective morphism with geometrically normal fibers. Then 
   $h^1(X_s, \o_{X_s})$ is locally constant by \cite[VI.2.7]{FGA}.
   This also holds if $g$ has  geometrically seminormal fibers.
 \end{exmp}

 \begin{exmp}\label{exmp.2}   Let $E$ be a locally free sheaf of rank $n$ over a scheme $Z$.
   Set $P:=\p_Z(E)$ with projection $\pi:P\to Z$.
   Then $\omega_{P/Z}\cong \o_{P/Z}(-n)\otimes \pi^*\det E$.
   Let $X\subset P$ be a member of $|\o_{P/Z}(n)|$. We have an exact sequence
   $$
   0\to \omega_{P/Z}\otimes \pi^*{\det}^{-1} E\to \o_P\to \o_X\to 0.
   $$
   Pushing it forward  gives that, for $n\geq 3$, 
   $$
   \pi_*\o_X\cong \o_Z,\qtq{and}
   \myR^{n-2}\pi_*\o_X\cong {\det}^{-1} E,
   $$
   and the others are 0.

   If $Z=C$ is a smooth, projective curve, then the Leray spectral sequence degenerates, and we get that
   $$
   H^{n-2}(X, \o_X)\cong H^0(C, {\det}^{-1} E),\qtq{and}
   H^{n-1}(X, \o_X)\cong H^1(C, {\det}^{-1} E).
   $$
   Consider now the case when
   $E=L\oplus\bigoplus_1^{n-1}\o_C$.

   If $H^0(C, L^m)=0$ for $0<m\leq n$, then sections of $\sym^n(E)$ come from the trivial summand,
   thus $X$ is a locally trivial family of cones in $\p^{n-1}$.
   A general $X$ is normal if $n\geq 3$, and its singular set is a section of
   $X\to C$.
Since  $\det E\cong L$, we get that
   $$
   H^{n-2}(X, \o_X)\cong H^0(C, L^{-1}),\qtq{and}
   H^{n-1}(X, \o_X)\cong H^1(C, L^{-1}).
   $$
   If $0\leq \deg L^{-1}\leq 2g(C)-2$, then the dimension of the cohomology groups $H^i(C, L^{-1})$ can change as $L$ varies.
   \end{exmp}

 These examples leave open the following.

 \begin{ques} Is there a flat, projective morphism
   $g:X\to C$ to a smooth curve such that
   \begin{enumerate}
   \item all fibers are  CM and reduced,
   \item the generic fiber is smooth, and
   \item  $h^1(X_c, \o_{X_c})$ is not constant.
   \end{enumerate}
   \end{ques}

 \begin{exmp}\label{h2.tor.exmp} Let $F$ be a smooth, projective surface, and $L$ a very ample line bundle with generating sections $s_1, s_2, s_3$.
   Set $$
   X:=(\tsum s_ix_i=0)\subset F\times \a^3_{\mathbf x}.
   $$
   Projection to $F$ shows that $X$ is an $\a^2$-bundle over $F$, hence smooth.   The fibers of the projection $\pi:X\to \a^3$ are 1-dimensional, save $\pi^{-1}(0,0,0)\cong F$.   Thus  $R^2\pi_*\o_X$ is supported at the origin, and there is a surjection $R^2\pi_*\o_X\onto H^2(F, \o_F)$.
 \end{exmp}

 \subsection*{Dualizing sheaves}{\ }
 \medskip

  In this section  we allow
 $S$ to be any noetherian scheme that has a dualizing complex
 \cite[\href{https://stacks.math.columbia.edu/tag/0A7B}{Tag
    0A7B}]{stacks-project}.

 \begin{say}\label{dual.elem.say}
   Let $g:X\to S$ be a flat, proper  morphism of schemes of pure relative dimension $n$, 
   with CM fibers.  Using Serre duality  over all Artinian subschemes of $S$ shows that if $\myR g^i_*\o_X$ is locally free, then it is dual to
   $\myR g^{n-i}_*\omega_{X/S}$. In particular
   \medskip

   \noindent{\it Claim \ref{dual.elem.say}.1.}  Under the above assumptions, $\myR g^i_*\o_X$ is locally free for every $i$ iff $\myR g^{i}_*\omega_{X/S}$ is locally free for every $i$.\qed
   \medskip

   If these hold, then Grothendieck duality gives the following.
     \medskip

   \noindent{\it Claim \ref{dual.elem.say}.2.}  $\myR g_*\o_X\simeq \sum_i\myR g^i_*\o_X[-i]$  iff $\myR g_*\omega_{X/S}\simeq \sum_i\myR g^i_*\omega_{X/S}[-i]$. \qed
 \end{say}

 Our aim is to establish a similar relationship between
 $\myR g_*\omega_{X/S}$ and $\myR g_*\omega_{X}$.
 The main result is Theorem~\ref{locally.omega.thm}, but we need some preliminaries.

\begin{lem}\label{lem:ltensor}
  Let $g:X\to S$ be a morphism of schemes that is flat and locally of finite
  presentation. Then
  \[
    \omega_X^\kdot\simeq\omega_{X/S}^\kdot\lotimes \myL g^*\omega_S^\kdot.
  \]
\end{lem}

\begin{proof}
  As in \cite[3.3.6]{Conrad00},  the exceptional inverse image functor can be written as 
  \[
    g^!(\blank ) \simeq \myR\sHom_X\left(\myL g^*\myR\sHom_S(\blank ,
      \omega_S^\kdot), \omega_X^\kdot\right).
  \]
  In particular,
  $$
    \omega_{X/S}^\kdot=g^!\o_S\simeq \myR\sHom_X\left(\myL
    g^*\omega_S^\kdot,\omega_X^\kdot\right).
    \eqno{(\ref{lem:ltensor}.1)}
  $$
  Then
  $$
  \begin{array}{rcl}
    \myR\sHom_X(\omega_{X/S}^\kdot\lotimes \myL
    g^*\omega_S^\kdot,\omega_X^\kdot)&\simeq& \myR\sHom_X(\omega_{X/S}^\kdot ,
    \myR\sHom(\myL
    g^*\omega_S^\kdot,\omega_X^\kdot))\\
    &\simeq&\myR\sHom_X(\omega_{X/S}^\kdot,
    \omega_{X/S}^\kdot)
    \end{array}
     \eqno{(\ref{lem:ltensor}.2)}
 $$
  by \cite[\href{https://stacks.math.columbia.edu/tag/08DJ}{Tag
    08DJ}]{stacks-project} and (\ref{lem:ltensor}.1).  Furthermore,
  $$
    \myR\sHom_X(\omega_{X/S}^\kdot, \omega_{X/S}^\kdot)\simeq\o_X    
    \eqno{(\ref{lem:ltensor}.3)}
 $$
  by \cite[\href{https://stacks.math.columbia.edu/tag/0E2V}{Tag
    0E2V}]{stacks-project}.
  Then  (\ref{lem:ltensor}.2--3)
    imply that 
  $$
    \omega_X^\kdot\simeq \myR\sHom_X(\o_X,\omega_X^\kdot)\simeq
    \myR\sHom_X\left(\myR\sHom_X(\omega_{X/S}^\kdot\lotimes \myL
      g^*\omega_S^\kdot,\omega_X^\kdot),\omega_X^\kdot\right).
      \eqno{(\ref{lem:ltensor}.4)}
  $$
  As $\omega_X^\kdot$ is a dualizing complex, (\ref{lem:ltensor}.4)
    implies the desired
  statement.
\end{proof}

\begin{cor}\label{cor.lem:ltensor}
  In addition to the assumptions in Lemma~\ref{lem:ltensor},  assume that  $X$ is CM. Then
  \[
    \omega_X\simeq\omega_{X/S}\otimes g^*\omega_S.
  \]
\end{cor}

\begin{proof}
  As $X$ is CM and $g$ is flat, it follows that $g$ is a CM morphism and $S$ is
  also CM by \cite[\href{https://stacks.math.columbia.edu/tag/045J}{Tag
    045J}]{stacks-project}. This implies that $\omega_X^\kdot\simeq\omega_X[\dim X]$
  and $\omega_S^\kdot\simeq\omega_S[\dim S]$. Also,  
  $\omega_{X/S}^\kdot\simeq\omega_{X/S}[\dim X-\dim S]$ by \cite[Thm.~3.5.1]{Conrad00}. As $g$ is flat, $g^*$ is
  exact, so $\myL g^*\simeq g^*$. Therefore Lemma~\ref{lem:ltensor} implies that
  $$
  \omega_X\simeq\omega_{X/S}\lotimes g^*\omega_S.
   \eqno{(\ref{cor.lem:ltensor}.1)}
  $$
  Finally, as the left hand side of (\ref{cor.lem:ltensor}.1) is a single sheaf, it follows
  that $\omega_{X/S}$ and $g^*\omega_S$ are Tor-independent, so
  $\omega_{X/S}\lotimes g^*\omega_S\simeq\omega_{X/S}\otimes g^*\omega_S$, and the
  statement follows.
\end{proof}

\begin{cor}\label{cor:projection-formula}
  Under the same assumptions as in Corollary~\ref{cor.lem:ltensor}, we have that
  $$
  \myR g_*\omega_X\simeq\myR g_*\omega_{X/S}\lotimes \omega_S.
  \eqno{(\ref{cor:projection-formula}.1)}
  $$
\end{cor}

\begin{proof}
  This is straightforward from  Corollaries~\ref{lem:ltensor}--\ref{cor.lem:ltensor} using
  the derived projection formula
  \cite[\href{https://stacks.math.columbia.edu/tag/08EU}{Tag 08EU}]{stacks-project}.
\end{proof}

 \begin{thm}\label{locally.omega.thm}
   Let $g:X\to S$ be a flat, proper  morphism of schemes of pure relative dimension $n$, 
   with CM fibers. Assume that $S$ is CM and has a dualizing sheaf $\omega_S$.  Then the following are
  equivalent.
  \begin{enumerate}
   \item $\myR^ig_*\omega_{X/S}$ is locally free for every $i$.
   \item  $\myR^ig_*\omega_{X}\simeq\omega_S\otimes\myR^ig_*\omega_{X/S}$,   and  $\myR^ig_*\omega_{X/S}$ is locally free for every $i$.
      \item  $\myR^ig_*\omega_{X}\simeq\omega_S\otimes E_i$   for every $i$, for some locally free sheaf $E_i$.
  \end{enumerate}
  Furthermore, if these hold, then the following are
  equivalent.
  \begin{enumerate}\setcounter{enumi}{3}
    \item  $\myR g_*\omega_{X/S}\simeq \sum_i\myR g^i_*\omega_{X/S}[-i]$.
 \item  $\myR g_*\omega_{X}\simeq \sum_i\myR g^i_*\omega_{X}[-i]$.
 \end{enumerate}
\end{thm}
 
 Proof. Note that $X$ is CM and has a dualizing sheaf $\omega_X\cong g^!\omega_S$ \cite[\href{https://stacks.math.columbia.edu/tag/0AA3
[stacks.math.columbia.edu]}{Tag0AA3}]{stacks-project}.

 For (\ref{locally.omega.thm}.1--3) we may assume that $S$ is affine. Assume that (\ref{locally.omega.thm}.1) holds.
 Then $\myR g_*\omega_{X/S}\simeq \sum_i\myR g^i_*\omega_{X/S}[-i]$ by Lemma~\ref{split.lem.2}, 
 and  (\ref{cor:projection-formula}.1) gives that
  $$
 \myR g_*\omega_X\simeq \omega_S\lotimes\myR g_*\omega_{X/S}
 \simeq \tsum_i \omega_S\otimes \myR g^i_*\omega_{X/S}[-i].
   $$
 This proves (\ref{locally.omega.thm}.2), which clearly implies
 (\ref{locally.omega.thm}.3).

 Next assume that (\ref{locally.omega.thm}.3) holds.
 Then $\myR g_*\omega_{X}\simeq \sum_i\myR g^i_*\omega_{X}[-i]$ by Lemma~\ref{split.lem.2}.   By \cite[3.3.6]{Conrad00}
 $$
 \omega_{X/S}
 \simeq \myR\sHom_X\left(\myL g^*\omega_S,\omega_X\right),
 $$
  and then
\cite[\href{https://stacks.math.columbia.edu/tag/08DJ}{Tag 08DJ}]{stacks-project}
 gives that
 $$
 \myR g_*\omega_{X/S}\simeq \myR\sHom_S\left(\omega_S,\myR
 g_*\omega_X\right)\simeq
 \tsum_i\myR\sHom_S\big(\omega_S,\myR^i g_*\omega_X[-i]\big)\simeq
 \tsum_i E_i[-i].
 $$
 Thus $\myR g^i_*\omega_{X/S}\cong E_i$, proving
 (\ref{locally.omega.thm}.1--2).

 Even if $S$ is not assumed affine, then the equivalence of (\ref{locally.omega.thm}.4--5) is implied by the isomorphisms
 $$
 \myR g_*\omega_X\simeq \omega_S\lotimes\myR g_*\omega_{X/S}
 \qtq{and}
 \myR g_*\omega_{X/S}\simeq \myR\sHom_S\left(\omega_S,\myR
 g_*\omega_X\right). \hfill \qed
 $$

 \begin{lem} \label{split.lem.2}
   Let $S$ be an affine scheme
   and  $R$  a bounded complex of quasi-coherent $\o_S$-modules.  
   Assume that
   \begin{enumerate}
   \item either the cohomology sheaves  $h^i(R)$ are all locally free,
   \item or   $S$ is CM,  has a dualizing sheaf $\omega_S$, and $h^i(R)\cong \omega_S\otimes E_i$   for every $i$, for some locally free sheaf $E_i$.
  \end{enumerate}
  Then   $R\simeq \sum h^i(R)[-i]$.
  \end{lem}

 Proof.   Let $d_i:R^i\to R^{i+1}$ be the differentials in $R$.
 If $h^i(R)$ is locally free, then it is projective, hence
 $\ker d_i\onto h^i(R)$ has a splitting $s_i: h^i(R)\to R^i$.
 These give the required $\sum h^i(R)[-i]\simeq R$.

 For (\ref{split.lem.2}.2),  we may assume that
   $R^j=0$ for $j<0$ and write $R$ as 
 $$
 0\to h^0(R)\to R\to R_+[-1]\stackrel{+1}{\longrightarrow}
 $$
 As in \cite[\href{https://stacks.math.columbia.edu/tag/07A9}{Tag 07A9}]{stacks-project}
 there is a spectral sequence  with $E_2$-page
 $$
 E_2^{i,j}=\sExt^i_S\bigl(h^{-j}R_+[-1], h^0(R)\bigr)
 \Rightarrow \sExt^{i+j}_S\bigl(R_+[-1], h^0(R)\bigr).
 $$
 If (\ref{split.lem.2}.2) holds then
 $$
 \sExt^m_S\bigl(h^{j}(R), h^i(R)\bigr)\cong
 \sExt^m_S\bigl(\omega_S, \omega_S\bigr)\otimes \sHom_S(E_j, E_i)
 $$
 is zero whenever $m\neq 0$.  Thus the 
 spectral sequence  degenerates and
 $R\simeq h^0(R)+R_+[-1]$. Induction gives that 
 $R\simeq \sum h^i(R)[-i]$. \qed

 \subsection*{Verifying Assumption~\ref{lf.thm.pf}.4}{\ }
 \medskip

 We start with another variant of Lemma~\ref{split.lem.2}.

\begin{lem}\label{lem:coh-of-ltwist-fg}
  Let $S$ be an affine scheme, and $R$ a complex of
  quasi-coherent $\o_S$-modules such that $R^j=0$ for $j<0$. Assume that $h^j(R)$ is locally free of finite rank for $j>c$. Further let $F$ be
  an $\o_S$-module. Then
  $$
  h^c(R\lotimes F)\simeq h^c(R)\otimes F.
  $$
  \end{lem}
  
\begin{proof}  Let $d_j:R^j\to R^{j+1}$ denote the differential.
  Then $h^j(R)$ is a quotient of $\ker d_j$. If $h^j(R)$ is projective,
  then $\ker d_j\onto h^j(R)$ has a splitting  $h^j(R)\to R^j$.
  Combining these gives a quasi-isomorphism
  $$
  R_c+\tsum_{j>c}h^j(R)[-j]\simeq R,
  \eqno{(\ref{lem:coh-of-ltwist-fg}.1)}
  $$
  where $R_c\leteq\tau_{\leq c}R$ denotes the canonical truncation of $R$ at degree
  $c$ \cite[\href{https://stacks.math.columbia.edu/tag/0118}{Tag  0118}]{stacks-project}.
  
 Therefore
  $h^j(R\lotimes F)\simeq h^j(R_c\lotimes F)$ for $j\leq c$.  On the other hand
  $h^c(R_c\lotimes F) \simeq h^c(R_c)\otimes F$, because $\otimes$ is right exact.
  \end{proof}

\begin{say}[Proof of  Assumption~\ref{lf.thm.pf}.4]\label{verify.assump}
  Using the notation of (\ref{lf.thm.pf}.4), consider the complex
  $\myR g_*\omega_{X/S}$. By (\ref{lf.thm.pf}.2) we know that its homologies are locally free of finite rank for $j>1$. Thus
  Lemma~\ref{lem:coh-of-ltwist-fg} implies that
  $$
  h^1\bigl(\myR g_*\omega_{X/S}\lotimes \omega_S\bigr)\cong
  h^1\bigl(\myR g_*\omega_{X/S}\bigr)\otimes \omega_S \cong
  \myR^1 g_*\omega_{X/S}\otimes \omega_S.
  $$
  On the other hand, (\ref{cor:projection-formula}.1)
 gives that 
  \[
    \myR g_*\omega_X\simeq\myR g_*\omega_{X/S}\lotimes \omega_S.
    \]
    Therefore
    $$
   \myR^1 g_*\omega_{X} \cong h^1\bigl(\myR g_*\omega_{X}\bigr)\cong
  \myR^1 g_*\omega_{X/S}\otimes \omega_S. \qed
  $$
\end{say}

 \subsection*{Allowing some non-rational singularities}{\ }
 \medskip

It turns out that  Theorem~\ref{lf.thm} also holds if  $X$ has a few non-rational singularities.

\begin{thm}\label{iso.nonrta.rem}
   Let $g:X\to S$ be a flat, projective morphism of varieties.
   Assume that  $X$ is CM, and there is 
   a closed subset $Z\subset X$ such that $Z\to S$ is finite, non-dominant, and 
    $X\setminus Z$ has rational singularities.
   Then  (\ref{lf.thm}.1--4) hold.
\end{thm}

Proof.  We start with checking (\ref{lf.thm}.1). We follow the proof in Paragraph~\ref{lf.thm.pf}.
We may assume that $S$ is local. Then a general very ample $H$ is disjoint from $Z$, and   has 
  rational singularities. Thus
  Theorem~\ref{lf.thm} applies to $H$.
  
    As before, (\ref{lf.thm.pf}.2) gives that 
    $\myR^{i+1}g_*\omega_{X/S}$ is locally free for $i\geq 1$.

    We have a potential problem with $\myR^{1}g_*\omega_{X}$.
    In applying (\ref{lf.thm.pf}.4), we needed to know that it is torsion free.  To understand it, let
    $r:X'\to X$ be a resolution of singularities.
    There is an injection  $r_*\omega_{X'}\into \omega_X$,
    whose quotient  $Q$ is supported on $Z$. This gives an exact sequence
    $$
    g_*Q \to  \myR^{1}(g\circ r)_*\omega_{X'}\to \myR^{1}g_*\omega_{X}
    \to \myR^{1}g_*Q.
    $$
    Here $\myR^{1}g_*Q=0$ since $Z\to S$ is finite.
    Also, $\myR^{1}(g\circ r)_*\omega_{X'} $ is torsion-free
    by \cite[2.1]{k-hdi1}. Since $Z$ does not dominate $S$, 
    $g_*Q \to  \myR^{1}(g\circ r)_*\omega_{X'}$ is the zero map.
Thus
$$
\myR^{i}(g\circ r)_*\omega_{X'}\cong \myR^{i}g_*\omega_{X} \qtq{for} i\geq 1,
\eqno{(\ref{iso.nonrta.rem}.1)}
$$
and the rest of the proof of (\ref{lf.thm}.1) in Paragraph~\ref{lf.thm.pf} works.
Thus we get that
$\myR^{i}g_*\omega_{X/S}$ is locally free for $i\geq 1$.
Therefore (\ref{locally.omega.thm}.1--3) all hold,  and so does (\ref{lf.thm}.4).

By \cite[p.172]{k-hdi2}  we know that
$$
\myR (g\circ r)_*\omega_{X'}\simeq \tsum \myR^i(g\circ r)_*\omega_{X'}[-i].
$$
Combining these with the natural maps
$g_*\omega_{X}\to \myR g_*\omega_{X}$  and $\myR (g\circ r)_*\omega_{X'}\simeq \myR g_*\omega_{X}$, we  get a map 
$$
g_*\omega_{X}+\tsum_{i\geq 1}\myR^i(g\circ r)_*\omega_{X'}[-i]
\to \myR g_*\omega_{X},
$$
which is an isomorphism on cohomologies by (\ref{iso.nonrta.rem}.1).  Thus 
$$
\myR g_*\omega_{X}\simeq \tsum \myR^ig_*\omega_{X}[-i].
\eqno{(\ref{iso.nonrta.rem}.2)}
$$
Thus (\ref{locally.omega.thm}.5) holds, and it implies
(\ref{locally.omega.thm}.4), thus (\ref{lf.thm}.3) holds.

The equivalence of (\ref{lf.thm}.3) and (\ref{lf.thm}.2) was noted in (\ref{dual.elem.say}.2).

\qed

\medskip
For the Picard scheme, this implies the following.
 
\begin{cor} \label{lf.cor.1.nr}
 Let $g:X\to S$ be a flat, projective morphism of relative dimension $n$.
   Assume that  $X$ is CM, and there is 
   a closed subset $Z\subset X$ such that 
   \begin{enumerate}
   \item     $X\setminus Z$ has rational singularities,
   \item each fiber of $Z\to S$ has dimension $\leq n-2$, and
   \item the generic fiber of $Z\to S$ has dimension $\leq n-3$.
   \end{enumerate}
   Then 
    \begin{enumerate}\setcounter{enumi}{3}
    \item  $\myR^1g_*\o_X$ is locally free and commutes with base change, and
      \item $\pics^\circ(X/S)\to S$ is a smooth algebraic group scheme, which commutes with  base change.
         \end{enumerate}
\end{cor}

Proof.  The claims are local on $S$, thus we may assume that $S$ is local.
Let $H\subset X$ be a general, sufficiently ample divisor.
If $n\geq 3$ then 
$\myR^1g_*\o_X\cong \myR^1(g|_H)_*\o_H$.

Using this repeatedly, we can reduce (\ref{lf.cor.1.nr}.4)
to the case $n=2$. Then the assumptions of Theorem~\ref{iso.nonrta.rem} hold, hence $\myR^1g_*\o_X$ is locally free and commutes with base change.

Finally (\ref{lf.cor.1.nr}.4) implies (\ref{lf.cor.1.nr}.5),  as we noted in the proof of Corollary~\ref{lf.cor.1}.\qed

 \begin{ack}  We thank  
  Y.-J.~Kim for explaining us the proof  of Corollary~\ref{lf.cor.1}, and B.~Bhatt for corrections.
  Partial  financial support   to JK  was provided   by the Simons Foundation   (grant number SFI-MPS-MOV-00006719-02).
  SK was partially supported by NSF Grant DMS-2100389.
    \end{ack}


 \def\cprime{$'$} \def\cprime{$'$} \def\cprime{$'$} \def\cprime{$'$}
  \def\cprime{$'$} \def\dbar{\leavevmode\hbox to 0pt{\hskip.2ex
  \accent"16\hss}d} \def\cprime{$'$} \def\cprime{$'$}
  \def\polhk#1{\setbox0=\hbox{#1}{\ooalign{\hidewidth
  \lower1.5ex\hbox{`}\hidewidth\crcr\unhbox0}}} \def\cprime{$'$}
  \def\cprime{$'$} \def\cprime{$'$} \def\cprime{$'$}
  \def\polhk#1{\setbox0=\hbox{#1}{\ooalign{\hidewidth
  \lower1.5ex\hbox{`}\hidewidth\crcr\unhbox0}}} \def\cdprime{$''$}
  \def\cprime{$'$} \def\cprime{$'$} \def\cprime{$'$} \def\cprime{$'$}
\providecommand{\bysame}{\leavevmode\hbox to3em{\hrulefill}\thinspace}
\providecommand{\MR}{\relax\ifhmode\unskip\space\fi MR }
\providecommand{\MRhref}[2]{%
  \href{http://www.ams.org/mathscinet-getitem?mr=#1}{#2}
}
\providecommand{\href}[2]{#2}

 \medskip
 
\end{document}